\documentclass[12pt,a4paper]{article}
\usepackage[margin=1in]{geometry}
\usepackage{amsmath,amsthm,amssymb}
\usepackage{enumitem}

\usepackage[T1]{fontenc}
\newtheorem{definition}{Definition}

\newtheorem{corollary}{Corollary}
\newtheorem{theorem}{Theorem}
\newtheorem{lemma}{Lemma}
\newtheorem{conjecture}{Conjecture}

\newtheorem{problem}{Problem}
\newtheorem{observation}{Observation}

\newtheorem*{theorem*}{Theorem}
\DeclareMathOperator{\conv}{conv}
\DeclareMathOperator{\md}{md}
\DeclareMathOperator{\ctd}{ctd}

\DeclareMathOperator{\pos}{pos}
\DeclareMathOperator{\cl}{cl}

\begin{document}
\setcounter{theorem}{-1}

\title{A characterization of the Carath\'eodory number for $H$-convexity}

\author{Vuong Bui\thanks{UET, Vietnam National University, Hanoi, 144 Xuan Thuy Street, Hanoi 100000, Vietnam (\texttt{bui.vuong@yandex.ru})\\
Part of the work was supported by the Deutsche Forschungsgemeinschaft
(DFG) Graduiertenkolleg ``Facets of Complexity'' (GRK 2434) at Institut f\"ur Informatik, Freie Universit\"at Berlin.}}
\date{}

\maketitle

\begin{abstract}
We show that the Carath\'eodory number for $H$-convexity is the maximum of two parameters: the Helly number for $H$-convexity and the cone number of $H$. The cone number in this article is defined as the maximal number of points of $H$ in conical position with an empty positive hull relative to the remaining points. Earlier partial results by Boltyanski and Martini can provide an exact value for the Carath\'eodory number only when the Helly number is $1$ or $2$.

We further establish connections between the Carath\'eodory numbers for $H$-convexity and that for $K$-strong convexity, where $H$ is the set of normals of $K$. Specifically, the Carath\'eodory number for $H$-convexity provides a lower bound for that of $K$-strong convexity. Moreover, if $K$ is a polytope, which has $|H|$ facets, then the Carath\'eodory number for $K$-strong convexity is at most the maximum of $|H|-1$ and the Carath\'eodory number for $H$-convexity. We conjecture a characterization of when the bound $|H|-1$ is attained. It is a consequence of a broader conjecture stating that the Carath\'eodory number for $K$-strong convexity is at most the maximum of the Carath\'eodory numbers for $H'$-convexity over all subsets $H'\subseteq H$. Finally, we observe that the Carath\'eodory number is at least the Helly number in any convex-structure where all sets are ordinarily convex.
\end{abstract}

\section{Introduction}
Convex sets, in the traditional sense, are the sets with the property that if two points are in a convex set then the segment connecting these points is also in the set. It is also known that a convex set is the intersection of a collection of halfspaces. Ordinary convexity can be extended by restricting the sets to intersect, e.g., by allowing only closed halfspaces of certain normals, or by allowing only translates of a certain convex body. The former case is known under the name $H$-convexity, while the latter one is often known as strong convexity. In this article, we provide Carath\'eodory's theorems for $H$-convexity and its connections to those for strong convexity.

$H$-convexity was introduced by Boltyanski (see \cite[Chapter III]{boltyanski2012excursions}) as follows.
\begin{definition}
Given a set of normals $H$, a set is said to be $H$-convex if it is an interesection of halfspaces of the form $\{x: \langle a,x\rangle \le b\}$ for a vector $a\in H$ and some real $b$.
\end{definition}
As a matter of convention, the normals in this paper are outer normal and they lie on the sphere $\mathbb S^{n-1}$ when we consider the convex sets in $\mathbb R^n$. We may explicitly scale them in certain cases.

Meanwhile, strong convexity arises in many sources (e.g. \cite{polovinkin1996strongly}) and it is also known under the name ``ball convexity'' (e.g. \cite{langi2013ball}).
\begin{definition}
Given a convex body $K$, a set is said to be $K$-strongly convex if it is the intersection of translates of $K$.
\end{definition}

The Helly theorem for the family of $H$-convex sets was given by Boltyanski (see \cite[Chapter III]{boltyanski2012excursions}), stating that the Helly number is the maximal number of vectors in $H$ being the vertices of a simplex containing the origin in its relative interior (which is at most the Helly number for ordinary convexity $n+1$). Recall that the Helly number for a family $\mathcal F$ of convex bodies is the least number $h$ so that if every $h$ members of $\mathcal F$ intersect, then all the members of $\mathcal F$ intersect. In fact, it was even shown that the Helly number does not decrease if we narrow down the family to only all translates of a convex body $K$ with the set of normals $H$ (see \cite[Chapter IV]{boltyanski2012excursions}). That being said we have the same Helly number for both $H$-convexity and $K$-strong convexity.

In this article, Carath\'eodory's theorem will be treated for these notions of convexity.
We first state Carath\'eodory's theorem in terms of the ordinary convexity.
\begin{theorem*}[The original Carath\'eodory theorem]
A point $p$ in the convex hull of a point set $X\subset \mathbb R^n$ is also in the convex hull of a subset $X'\subset X$ of at most $n+1$ points.
\end{theorem*}

Like ordinary convexity, the $H$-convex hull $\conv_H X$ (resp. $K$-strongly convex hull $\conv_K X$) of a set $X$ is defined to be the minimal $H$-convex set (resp. $K$-strongly convex hull) containing $X$. 
Note that we always assume $X$ is contained in a translate of $K$, as $\conv_K X$ is undefined otherwise.
Also note that there should not be a confusion between the notations $\conv_H$ and $\conv_K$, as throughout the paper $K$ is always a convex body and $H$ is a set of normals.

Replacing the ordinary convex hull by the two recently defined convex hulls, we study the following problem.
\begin{problem}\label{prob:main}
Given a closed set of vectors $H$ (resp. a convex body $K$), what is the smallest number $h$ (resp. $k$) such that: For every closed set $X$,
    \[
	\conv_H X = \bigcup_{X': X'\subseteq X, |X'|\le h} \conv_H X'
    \]
(resp. $\conv_K X = \bigcup\limits_{X': X'\subseteq X, |X'|\le k} \conv_K X'$)?
\end{problem}

When there is no such number $h$ (resp. $k$), we say $h=\infty$ (resp. $k=\infty$).
Note that the condition on the closedness of $X$ is necessary due to the closedness of the halfspaces and the translates of $K$: For example, consider $H=\{(1,0,0)\}\subset\mathbb S^2$ contains a single unit vector following the $x$-direction in the $3$-dimensional space and the open set $X=\{(x,0,0): x<0\}$, then $\conv_H X=\{(x,y,z): x\le 0\}$ but no finite subset $X'\subset X$ has $\conv_H X'$ containing the origin $0$. (A small remark is that $\conv_H X$ contains $\cl X$ since $H$-convex sets are closed.)

The condition on the closedness of $H$ is however for convenience only. The reason is that $\conv_H X=\conv_{\cl H} X$ for a closed $X$ (see \cite[Theorem $20.4$]{boltyanski2012excursions}).

The so-called \emph{Carath\'eodory number $k$ for $K$-strong convexity} is shown in \cite{bui2021caratheodory} to be at least the dimension of $K$ and the bound is attained for a certain class of convex bodies. On the other hand, the number can be $\infty$ when $K$ is a cone in $\mathbb R^3$ with the base being a disk. Another result \cite{karasev2001characterization} is that this number is at most $n+1$ for a generating convex body in $\mathbb R^n$, with some discussion in Section \ref{sec:relations}.

Although the Carath\'eodory number for strong convexity has been established for only a few certain cases and seems to be a hard problem in general, one may hope for a solution to the so-called \emph{Carath\'eodory number $h$ for $H$-convexity}. Boltyanski and Martini gave the following characterization.
\begin{theorem}[Boltyanski and Martini 2001 \cite{boltyanski2001caratheodory}]
    For a set $H\subseteq \mathbb S^{n-1}$ of normals that is not one-sided, we have the following conclusions 
    \begin{center}
        \begin{tabular}{|c|c|c|}
            \hline
             condition on $n$ & condition on $\md H$ & conclusion on $\ctd(H)$ \\ 
            \hline
             $n\ge 2$ & $\md H=1$ & $\ctd(H)=n-1$ \\ 
             $n=2$ & $\md H=2$ & $\ctd(H)=2$ \\  
             $n\ge 3$ & $\md H=2$ & $n-1\le\ctd(H)\le n^2-n-1$\\    
             $n\ge 4$ & $3\le \md H\le n-1$ & $n-1\le\ctd(H)\le \infty$ \\    
             $n\ge 3$ & $\md H=n$ & $n\le\ctd(H)\le \infty$ \\    
            \hline
        \end{tabular}
    \end{center}
    where $\md H$ is one less than the Helly number for $H$-convexity and $\ctd(H)$ is one less than the Carath\'eodory number for $H$-convexity.
\end{theorem}
We use the notations $\md H$ and $\ctd(H)$ of \cite{boltyanski2001caratheodory} in this theorem only for easier reference. Note that the authors of \cite{boltyanski2001caratheodory} are a bit strict with constraining the set $H$ to be not one-sided to address the set of normals of convex bodies. However, it is a minor issue that does not affect the results.

Although the lower and upper estimates are exact in the sense that there are examples available, we do not quite have an idea of the value of the Carath\'eodory number when $n\ge 3$. Therefore, we give a more definite characterization by introducing the following definition.
\begin{definition}
The cone number of a set $H\subseteq \mathbb S^{n-1}$ is the maximal number of vectors in $H$ such that they are in \emph{conical position} (i.e., all points are strictly separable from $0$ and no point is in the positive hull of the rest) and their positive hull is free of any remaining vector of $H$. If there exist arbitrarily many such vectors, the cone number is defined to be $\infty$.
\end{definition}

If $H$ is finite, the cone number can be interpreted as the maximal degree of a vertex in a polyhedron with the set of normals $H$.

The Carath\'eodory number for $H$-convexity can be characterized in the following main result of the article.
\begin{theorem} \label{thm:characterize}
  The Carath\'eodory number for $H$-convexity is the maximum of two parameters: the Helly number for $H$-convexity and the cone number of $H$.
\end{theorem}
Note that if we do not assume $H$ to be closed as in the statement of Problem \ref{prob:main}, the Helly number is still the same for both $H$-convexity and $(\cl H)$-convexity (see \cite[Theorems $17.3$ and $22.1$]{boltyanski2012excursions}) but the cone number may be different for $H$ and $\cl H$. Indeed, while $\mathbb S^2$ is the closure of $(\mathbb S^2\setminus\pos\{(1,0,0),(0,1,0)\})\cup\{(1,0,0),(0,1,0)\}$, their cone numbers are respectively $1$ and $2$. 

The fact that the Helly number for $H$-convexity is a lower bound on the Carath\'eodory number for $H$-convexity is already observed by Boltyanski and Martini in \cite[Theorem $1$]{boltyanski2001caratheodory}. It is actually a case of Theorem \ref{thm:cara>helly}. The main merit of Theorem \ref{thm:characterize} is the introduction of the cone number, which finishes the characterization. The notion is not artificial, as we can easily use it to prove related results as in Section \ref{sec:relations}. The proof of Theorem \ref{thm:characterize} in Section \ref{sec:characterize}, although being technical, is not overly complicated, due to the simplicity of the concepts. 

It should be noticed that more works are still to be done, as this is only an ``algebraic characterization'' of the Carath\'eodory number. A similar scenerio is that while the Helly number for $H$-convexity can be decribed as simple as the maximal number of points in $H$ being the vertices of a simplex whose relative interior contains $0$, a more thorough geometric description of the convex bodies with the set of normals $H$ is largely open, and known under the name ``the Sz\H{o}kefalvi--Nagy problem'' \cite{boltyanski2012excursions}. We write the problem for the Carath\'eodory number as below.
\begin{problem}
\label{prob:geometric-char}
    Give a geometrical description of the vector system $H$ (or those polyhedra with the set of normal $H$) so that the Carath\'eodory number is some given number $r$.
\end{problem}

Whether Problem \ref{prob:geometric-char} heavily relies on a solution of the Sz\H{o}kefalvi--Nagy problem is not so obvious. Suppose we restrict ourselves to one-sided $H$, for which the polyhedra are polytopes. One can see that the cone number is at least $n$ while the Helly number is at most $n+1$. Therefore, the Helly number can decide the Carath\'eodory number only when the Helly number is $n+1$ and the cone number is $n$. 
Note that the cone number $n$ is equivalent to the condition that if $n+1$ normals are in conical position, then their positive hull contain another normal from $H$. The polytopes with such condition were in turn shown to be exactly the monotypic polytope in \cite[Theorem $1.5$]{bui2025every} (the monotypic polytopes were defined in \cite{mcmullen1974monotypic}).
Note that monotypic polytopes include strongly monotypic polytopes, whose sets of normals do not contain $n+1$ normals in conical position, see \cite[Theorem $1.6$]{bui2025every}. Unfortunately, even for this subclass of strongly monotypic polytopes, the characterization is already nontrivial, see \cite{borowska2008strongly} for an attempt in $\mathbb R^3$. In conclusion, one can follow the roadmap below to attempt to solve Problem \ref{prob:geometric-char}, none of the steps is an easy task:
\begin{itemize}[itemsep=0pt]
	\item Characterize monotypic polytopes together with their Helly numbers.
	\item For those non-monotypic polytopes, characterize the cone numbers.
\end{itemize}

Those who prefer to read the proof of the main result, Theorem \ref{thm:characterize}, may go to Section \ref{sec:characterize}. Otherwise, we proceed to the next section, which relates the main result to others.

\section{Connections between the Carath\'eodory numbers for strong convexity, $H$-convexity, and related notions}
\label{sec:relations}
As a relaxed version of strong convexity, $H$-convexity introduces the following bound.
\begin{theorem}
\label{thm:lower-bound}
For a convex body $K$ with the set of normals $H$, the Carath\'eodory number for $K$-strong convexity is at least the Carath\'eodory number for $H$-convexity.
\end{theorem}
A proof given in Section \ref{sec:lower-bound} uses an observation that the more we scale down a point set, the less different its convex hulls for $K$-strong convexity and $H$-convexity are. Thinking differently but still of the same nature, we keep the point set and scale up $K$. When we increase the radius of a ball $K$, the $K$-strongly convex hull approaches the ordinary convex hull, while the $H$-convex hull is identical to the ordinary convex hull (we assume the point set is closed). 

Note that the lower bound provided by the Carath\'eodory number for $H$-convexity is at least as good as the known bound by the dimension $n$ of $K$ (the latter bound was given in \cite[Theorem $1.8$]{bui2021caratheodory}). The readers may notice that the technique in the proof of Theorem \ref{thm:lower-bound} can be used for the proof in \cite{bui2021caratheodory} instead of the argument with John's ellipsoid. In particular, several separating hyperplanes will be used instead of a single separating ball. In \cite[Section $5$]{bui2021caratheodory}, the cone was provided as an example with an infinite Carath\'eodory number for $K$-strong convexity. A more important consequence of this work is now we have a class of examples with an infinite Carath\'eodory number for $K$-strong convexity: convex bodies $K$ whose set of normals $H$ has an infinite cone number.

The number of facets of a polytope $K$ can be seen as a trivial upper bound on the Carath\'eodory number for $K$-strong convexity. One can slightly improve the bound by combining it with the Carath\'eodory number for $H$-convexity.
\begin{theorem}
\label{thm:trivial-upper-bound}
For a polytope $K$, if $k$ is the Carath\'eodory number for $K$-strong convexity and $h$ is the Carath\'eodory number for $H$-convexity where $H$ is the set of normal vectors of $K$, then
\[
    k\le \max\{h, |H|-1\}.
\]
\end{theorem}

Combining the theorems, one obtains the following result for pyramids.\footnote{A pyramid in $\mathbb R^n$ with respect to a convex body $K'\in\mathbb R^{n-1}$ is the convex hull of $K'\cup\{p\}$ for a point $p$ not in the space of $K'$.}
\begin{corollary}
\label{cor:pyramids}
If $K\subset\mathbb R^n$ has the same set of normals as a pyramid that is not a simplex, then the Carath\'eodory number for $K$-strong convexity is the number of relative facets of the base.
\end{corollary}
\begin{proof}
By the characterization in Theorem \ref{thm:characterize}, the Carath\'eodory number for $H$-convexity is the number of relative facets of the base. Therefore, by the lower bound in Theorem \ref{thm:lower-bound}, the Carath\'eodory number for strong convexity is at least $|H|-1$. By the upper bound in Theorem \ref{thm:trivial-upper-bound}, this Carath\'eodory number is precisely $|H|-1$.
\end{proof}

Note that when the base is an $(n-1)$-simplex, the considered pyramid is actually an $n$-simplex, which is known to have the Carath\'eodory number $n+1$ for strong convexity, as one can easily verify. It turns out that pyramids (including simplices) have the same Carath\'eodory number for both strong convexity and $H$-convexity. However, the $K$-strongly convex hull and the $H$-convex hull for an arbitrary point set are in general different (e.g., the base of the pyramid is a regular pentagon and the point set contains two points attaining the diameter of the base). In fact, $\conv_H X\subseteq \conv_K X$. The equality is attainable, e.g. for cubes regardless of the set $X$.

The difference between the Carath\'eodory number for $K$-strong convexity and the number of facets of $K$ can vary. They are equal for simplices, as easily verified. In fact, simplices are the only polytopes with the Carath\'eodory number equal to the number of facets.
\begin{corollary}
\label{cor:simplices}
Simplices are the only polytopes with the Carath\'eodory number for strong convexity equal to the number of facets.
\end{corollary}
\begin{proof}
By Theorems \ref{thm:lower-bound} and \ref{thm:trivial-upper-bound}, if the Carath\'eodory number for strong convexity is the same as the number of facets, then the Carath\'eodory number for $H$-convexity is equal to the Carath\'eodory number for $K$-strong convexity. By Theorem \ref{thm:characterize}, the Carath\'eodory number for $H$-convexity is the maximum of two numbers. The latter number, the highest vertex degree of a polytope taking $H$ as the set of normals, cannot be the same as the number of facets. Therefore, the former number, the maximum number of points in $H$ being the vertices of a simplex containing $0$ in its relative interior, must be the same as the number of facets. It happens only when the considered polytope is a simplex.
\end{proof}

Giving an extra facet to a simplex does not increase the Carath\'eodory number for strong convexity, which is again a corollary of the same type.
\begin{corollary}
\label{cor:simplices-with-extra-bounding-facet}
For a polytope $K$ obtained from bounding an $n$-simplex by an extra facet, the Carath\'eodory number for $K$-strong convexity is $n+1$.
\end{corollary}
\begin{proof}
By Theorem \ref{thm:lower-bound}, the Carath\'eodory number is at least $n+1$, which is the Carath\'eodory number for $H$-convexity (by the characterization as in Theorem \ref{thm:characterize}). On the other hand, it is also at most $n+1$ as confirmed by Theorem \ref{thm:trivial-upper-bound}.
\end{proof}

We have given examples where the Carath\'edory number for strong convexity is the same or one less than the number of facets. In fact, the difference can be much larger for other cases. For example, in the case of the cube $K=[0,1]^n$, both the Carath\'eodory numbers are equal to $n$ and that is only half of the number of facets $2n$ (see \cite[Theorem $1.6$]{bui2021caratheodory}). If we fix the difference between the Carath\'eodory number for strong convexity and the number of facets, one may ask the following question.
\begin{problem}
\label{prob:characterize-for-t}
For $t\ge 0$, what are exclusively the polytopes $K$ that have the Carath\'eodory number for strong convexity exactly $t$ less than the number of facets.
\end{problem}

The case $t=0$ is treated in Corollary \ref{cor:simplices}. The case $t=1$ would be fully characterized if the following conjecture holds.
\begin{conjecture} \label{conj:|H|-1}
If the Carath\'eodory number for $K$-strong convexity is one less than the number of facets of a polytope $K$, then the normals of $K$ are the same as either a simplex bounded by an extra facet, or a pyramid that is not a simplex.
\end{conjecture}

The other direction is actually the content of Corollaries \ref{cor:pyramids} and \ref{cor:simplices-with-extra-bounding-facet}. Note that these classes of polytopes are not the same: If a simplex has the normals $v_0,v_1,v_2,v_3$ and we bound it by an extra facet of normal $-v_0$, then the result is not a pyramid.

In these corollaries, the Carath\'eodory number is completely decided by the set of normals. However, it seems to be not always the case.
For example, let $K$ be a pyramid in $\mathbb R^3$ with the base being a regular $m$-gon. 
Suppose we bound the pyramid by an extra bounding facet parallel to the base. 
If the new polytope $K'$ has almost the same height as the original pyramid (by letting the bounding facet so close to the apex), the Carath\'eodory number for strong convexity remains the same. Roughly speaking, $K'$ still looks like a pyarmid and one may follow the same kind of proof as in \cite[Section $5$]{bui2021caratheodory}. However, it seems that when we move the bounding facet downward to the base, the Carath\'eodory number would decrease at some point. The intuition is that when the height is so small (relatively to the base), the polytope looks like a $2$-dimensional $m$-gon and the Carath\'eodory number for strong convexity should not be large (regardless of $m$) and even bounded (note that the Carath\'edory number for a $2$-dimensional convex body is at most $3$). However, it seems to be hard to decide precisely the Carath\'eodory numbers for particular heights or decide at which height the number decreases. Such an argument should be very clumsy and we avoid that.
Note that we want the polygon to be regular to avoid some degenerate case, e.g., an $m$-gon for $m>4$ but still looking like a square.

The polytope $K'$ constructed above serves as an example where the Carath\'eodory number for strong convexity is $2$ less than the number of facets. It suggests that Problem \ref{prob:characterize-for-t} with $t=2$ is in general not easy to characterize, as the relative positions of facets matter. It also suggests the following conjecture.
\begin{conjecture} \label{conj:untrivial-upper-bound}
Let $h(H)$ denote the Carath\'eodory number for $H$-convexity. If the Carath\'eodory number for $K$-strong convexity is $k$ for a polytope $K$ with the set of normals $H$, we have
\[
    h(H)\le k\le \max_{H'\subseteq H} h(H').
\]
\end{conjecture}
The lower bound is the content of Theorem \ref{thm:lower-bound} while the upper bound is still in question. It can be seen as a better upper bound than the one in Theorem \ref{thm:trivial-upper-bound} as the number of facets is usually much larger. Stating differently, the upper bound is the maximum of two parameters: the maximal number of normals being the vertices of a simplex whose relative interior contains $0$, and the maximal number of normals being in conical position. The latter one is a relaxed version of the cone number. Stating the upper bound in this way, we can see that Conjecture \ref{conj:untrivial-upper-bound} implies Conjecture \ref{conj:|H|-1}.
\begin{observation}
    Conjecture \ref{conj:untrivial-upper-bound} implies Conjecture \ref{conj:|H|-1}.
\end{observation}
\begin{proof}
    Assume Conjecture \ref{conj:untrivial-upper-bound} holds. Consider a polytope $K$ so that the Carath\'eodory number for $K$-strong convexity is $|H|-1$.
    Due to Theorem \ref{thm:characterize}, if we have $h(H')=|H|$ for some $H'\subseteq H$, then $K$ is a simplex, which is obviously not the case. Therefore, $\max\limits_{H'\subseteq H} h(H') = |H|-1$ by the upper bound in Conjecture \ref{conj:untrivial-upper-bound}. By Theorem \ref{thm:characterize}, we either have $|H|-1$ vectors in $H$ being the vertices of a simplex whose relative interior contains $0$, or we have $|H|-1$ vectors in $H$ in conical position. The only remaining vector in $H$ is the normal of either an extra facet to a simplex or the base of a pyramid.
\end{proof}

Conjecture \ref{conj:untrivial-upper-bound} also implies a known result that the Carath\'eodory number for $K$-strong convexity is at most $n+1$ for $n$-dimensional polytopes $K$ with the generating property. The generating property is defined as follows.
\begin{definition}
A convex body $K$ is said to be generating if for every set of translate vectors $T$, either the intersection of the translates of $K$ by $T$ is empty or we can find a convex body $K'$ such that
\[
K' + \bigcap_{t\in T} (K-t) = K.
\]
\end{definition}

In fact, it was proved in \cite{karasev2001characterization} (in Russian, see \cite{holmsen2017colorful} for a version in English) that the Carath\'eodory number for $K$-strong convexity is at most $n+1$ for any $n$-dimensional convex body $K$ with the generating property, using topological arguments. In \cite{bui2025every}, the generating property for polytopes was shown to be identical to the notion of strong monotypy for polytopes, which in turn can be used to characterize the set of normals as follows.
\begin{theorem*}[Characterization of polytopes with the generating property \cite{bui2025every}]
    The following two conditions are equivalent for an $n$-dimensional polytope $P$:
    \begin{itemize}
        \item $P$ has the generating property
        \item Every $n+1$ normals of $P$ are not in conical position.
    \end{itemize}
\end{theorem*}

\begin{corollary}
    If Conjecture \ref{conj:untrivial-upper-bound} holds, then we can directly deduce that the Carath\'eodory number for $K$-strong convexity is at most $n+1$ for an $n$-dimensional polytope $K$ with the generating property.
\end{corollary}
\begin{proof}
    Let $H$ be the set of normals of $K$. By the characterization of the polytopes with the generating property, the cone number of any subset $H'\subseteq H$ is at most $n$. Meanwhile, the Helly number for any subset $H'$ is at most $n+1$. Therefore, if Conjecture \ref{conj:untrivial-upper-bound} holds, then the Carath\'eodory number for $K$-strong convexity is at most $\max\{n,n+1\}=n+1$.
\end{proof}

It seems that one can conclude something stronger since the above proof points out that the cone number is only at most $n$ (instead of $n+1$). We relate it to the result \cite[Theorem $1.5$]{bui2021caratheodory} as follows.
\begin{theorem*}[\cite{bui2021caratheodory}]
    Given a generating set $K$ in $\mathbb R^n$. If a point $p$ is in the $K$-strongly convex hull of a set $X$ but not in the ordinarily convex hull of $X$, then $p$ is in the $K$-strongly convex hull of a subset of $X$ with at most $n$ points.
\end{theorem*}

Following this direction, and together with the lower bound from Theorem \ref{thm:lower-bound}, one may start thinking of characterizing which polytopes with the generating property (i.e., strongly monotypic polytopes) have the Carath\'eodory number $n$ for strong convexity and which ones have the number $n+1$. We leave this as an open problem. It may be related to a geometric characterization of the strongly monotypic polytopes, which has been done in $\mathbb R^3$ only \cite{borowska2008strongly}.

As we may observe, the Carath\'eodory number is at least the Helly number for ordinary convexity as well as for $H$-convexity and $K$-strong convexity. However, this is known to be not true for the following general notion of convexity (see \cite{hammer1960kuratowski} for an example).

\begin{definition}
A family of subsets of a ground set $E$ is said to be a convex-structure $\mathcal C$ if $\emptyset, E$ are in the family and the family is closed under intersection.\footnote{Note that the usual definition for $K$-strong convex hull of a set $X$ is undefined when $X$ cannot be contained in a translate of $K$. However, we can set the hull to be $E=\mathbb R^n$ in this case to make strong convexity to be a convex-structure. This modification does not affect the Carath\'eodory number since the Carath\'eodory number of points in $\conv_K X$ for $X$ not contained in any translate of $K$ is actually the Helly number (as noted in \cite[Remark 1.10]{bui2021caratheodory}). See also Theorem \ref{thm:cara>helly}.} Each set in the family is called a $\mathcal C$-convex set. The Helly number and the Carath\'eodory number for $\mathcal C$-convexity are defined accordingly.
\end{definition}

Nonetheless, if every set in $\mathcal C$ is (ordinarily) convex in a Euclidean space, then the relation holds, as in the following theorem.
\begin{theorem}
\label{thm:cara>helly}
If every set in a convex-structure is (ordinarily) convex in a Euclidean space, then the Carath\'eodory number is at least the Helly number for the convexity of this structure.
\end{theorem}
We prove the theorems in the following sections.

\section{Proof of the main result}
\label{sec:characterize}
In this section, we prove Theorem \ref{thm:characterize}, the main result of the article.

Let $X$ be a finite set such that $0\in\conv_H X$ but $0\notin\conv_H X'$ for any proper subset $X'\subset X$.
For convenience, we name the elements of $X$ by $x_i$ and the elements of $H$ by $a_i$, where the index $i$ can be replaced by other variables.
  The two properties can be understood as:\footnote{In case the readers may find the writing of $\exists f(j)$ in Property $2$ not very conventional, one can rewrite it as $\forall j\;\exists i=f(j),\; \left(\langle a_{i}, x_j\rangle\ge 0;\quad \forall j'\ne j,\;\langle a_{i}, x_{j'}\rangle < 0\right)$.}
  \begin{itemize}
    \item $0\in\conv_H X$ means
    \begin{equation} \label{covering}
        \forall i\;\exists j,\; \langle a_i, x_j\rangle \ge 0.
    \end{equation}
    \item
    $\forall X'\subset X,\; 0\notin\conv_H X'$ means 
    \begin{equation} \label{excluding}
        \forall j\;\exists f(j),\; \left(\langle a_{f(j)}, x_j\rangle\ge 0\quad\text{and}\quad \forall j'\ne j,\;\langle a_{f(j)}, x_{j'}\rangle < 0\right).
    \end{equation}
  \end{itemize}
  Property \eqref{covering} states that the set $H$ is covered by the family of halfspaces $\{a: \langle a, x_j\rangle\ge 0\}$ for every $x_j\in X$, while Property \eqref{excluding} states that each halfspace of the family \emph{exclusively} contains a point of $H$ (i.e., each halfspace contains a point of $H$ that no other halfspace contains).
  
  We consider the following two cases. We show in the first case that the cardinality $|X|$ is at most the Helly number for $H$-convexity and in the second case that $|X|$ is at most the cone number of $H$. These two facts confirm that the maximum of the Helly number and the cone number is an upper bound of the Carath\'eodory number. We will later show that it is also a lower bound to finish the proof. The two cases are:

  \begin{itemize}[leftmargin=*]
    \item
   $0\in\conv\{a_{f(j)}\}$.

      We show that $\{a_{f(j)}\}$ must be the vertices of a simplex with $0$ in its relative interior, which means that $|X|$ in this case is at most the Helly number for $H$-convexity.
      Indeed, suppose $0$ is in the convex hull of a proper subset, say $\sum\limits_{j=1}^{\ell} \lambda_j\, a_{f(j)}=0$ where $1<\ell<|X|$, all $\lambda_j$ are positive. Multiplying both sides by $x_k$ for $k>\ell$, we have $\sum\limits_{j=1}^{\ell} \lambda_j\, \langle a_{f(j)}, x_k\rangle=0$. The left hand side is negative due to $\langle a_{f(j)}, x_k\rangle < 0$ for every $j<k$. The contradiction means that $0$ is not in the convex hull of any proper subset. 
      It follows that $\{a_{f(j)}\}$ are the vertices of a simplex whose relative interior contains $0$.

  \item
  $\{a_{f(j)}\}$ is separable from $0$ by a halfspace $\{a: \langle \vec{n}, a\rangle > 0\}$.

    We first manipulate\footnote{This way of manipulation is suggested by Roman Karasev in a slightly different formulation (private communication).} the halfspaces to achieve a stronger form for Property \eqref{excluding} while keeping Property \eqref{covering}. This is a key step to obtain a new subset $\{a'_{f(j)}\}$ of $H$ whose cardinality is at most the cone number.

    For a given $j$, we rotate the halfspace $\{a: \langle a, x_j\rangle\ge 0\}$ by decreasing $x_j$ by $\alpha\, \vec{n}$ ($\alpha \ge 0$) as much as possible so that the two properties remain valid. Since there are points of $H$ that are exclusively contained in this halfspace, e.g. $a_{f(j)}$ by Property \eqref{excluding}, we cannot decrease after a certain $\alpha$ (as otherwise we will break Property \eqref{covering} by not covering all the points). 
    Indeed, let the set of the points exclusively contained in the halfspace be $A$. 
    For every $a_i\in A$, we need
    \[
        \langle a_i,x_j - \alpha\, \vec{n}\rangle \ge 0 \iff \langle a_i, x_j\rangle\ge \alpha\, \langle \vec{n}, a_i\rangle \iff \alpha\le \frac{\langle a_i, x_j\rangle}{\langle \vec{n}, a_i\rangle}.
    \]
    Setting $\alpha=\inf\limits_{a_i\in A} \frac{\langle a_i, x_j\rangle}{\langle \vec{n}, a_i\rangle}$, we obtain
    the new halfspace containing a point $\hat{a}\in A$ on the supporting hyperplane. Note that $\hat{a}$ is a point that attains the infimum, due to the closedness of $H$. 
    We replace $a_{f(j)}$ by $\hat{a}$ and replace $x_j$ by $x_j - \alpha\, \vec{n}$. This new setting for Property \eqref{excluding} has a nicer feature that $\langle a_{f(j)}, x_j\rangle = 0$ (rather than inequality in the original) but still keeps the old feature $\langle a_{f(j)}, x_{j'}\rangle < 0$ for $j'\ne j$.

  We modify the halfspaces and take the point on each hyperplane by the above method sequentially for $j=1, \ldots, |X|$, and eventually arrive at a new property in the place of Property \eqref{excluding}:
\[
  \forall j\;\exists f(j),\; \left(\langle a_{f(j)}, x_j\rangle = 0\quad\text{and}\quad \forall j'\ne j,\;\langle a_{f(j)}, x_{j'}\rangle < 0\right),
\]
  while still satisfying Property \eqref{covering}.

  These points $\{a_{f(j)}\}$ are in conical position, since every point $a_{f(j)}$ is separated\footnote{When we say $X$ is separated from $Y$ by a set $Z$ and $Z$ is not a hyperplane, we mean $Z$ contains $X$ and is disjoint from $Y$.} from the rest by $\{a: \langle a, x_j\rangle \ge 0\}$. Moreover, their positive hull contains no other point of $H$. Indeed, suppose there is a point $a'\in H$ other than $\{a_{f(j)}\}$ in the positive hull, we have $a'=\sum\limits_{j=1}^k \lambda_j\, a_{f(j)}$ where $\lambda_j$ are all nonnegative with at least two of them being positive. It means $a'$ is not in the halfspace of any $x_i$, since $\langle a', x_i\rangle =\sum\limits_{j=1}^k \lambda_j\,\langle a_{f(j)},x_i\rangle$ which has at least one of the terms $\lambda_j\, \langle a_{f(j)}, x_i\rangle$ negative while others are nonpositive. This is a contradiction to Property \eqref{covering} of covering $H$ by the halfspaces.

  These normals $\{a_{f(j)}\}$ are separable from $0$, in conical position and no other normal is in its positive hull, so we can conclude that their cardinality is not greater than the cone number of $H$.
  \end{itemize}
  
  Due to the analysis of the two cases, the maximum of the Helly number and the cone number is an upper bound of the Carath\'eodory number.

  It remains to show that it is also a lower bound. Indeed, consider the following two cases:
  \begin{itemize}[leftmargin=*]
  \item
    If the maximum is the Helly number with some $k$ normals $B=\{a_1,\ldots,a_k\}$ being the vertices of a simplex whose relative interior contains the origin, we positively scale each $a_i$ so that $a_1 + \dots + a_k = 0$ and choose $X=\{x_1,\ldots,x_k\}$ in the linear span of $B$ so that $\langle a_j, x_i\rangle = -1$ for every $i\ne j$, which in turn means $\langle a_i, x_i\rangle = k-1$. A feature of $X$ is that $0\in\conv X$ since $x_1 + \dots + x_k$ is orthogonal to every $a_i$, i.e. $x_1 + \dots + x_k = 0$.
    
    Because $0\in\conv X$, it follows that $0\in\conv_H X$. Also, $0\notin\conv_H X'$ for any set $X'$ obtained by dropping any $x_i$ since $\langle a_i, x_j\rangle <0$ for every $a_j\in X'$, which has $j\ne i$.

    Therefore, $k$ is a lower bound of the Carath\'eodory number.
    
  \item
    If the maximum is the cone number with a set $B\subseteq H$ of maximal cardinality (if $B$ is infinite, we assume $B$ is maximal up to inclusion) such that $B$ is separable from $0$, in conical position with the positive hull not containing any other point of $H\setminus B$, then we consider the set $X$ of the points $x_i$ corresponding to each $a_i\in B$ such that $\langle a_i, x_i\rangle =0$ for all $i$ while $\langle a_i, x_j\rangle < 0$ for every $j\ne i$. 
    The set $X$ exists due to the conical position of the normals in $B$, where every point is separable from the rest.

    The problem of showing $X$ is an example confirming the bound can be reduced to showing that: For any such $X$, the family of halfspaces $\{a: \langle a, x_i\rangle \ge 0\}$ for each $a_i\in B$ cover the whole $H$. Indeed, assuming it, we have $0\in\conv_H X$ since for every $a\in H$, there is $i$ such that $\langle a, x_i\rangle \ge 0$ due to covering of the halfspaces. Dropping any $x_i$ from $X$ gives a subset $X'$ satisfying $0\notin \conv_H X'$ since $\langle a_i, x_j\rangle  < 0$ for every $j\ne i$. 

    It remains to prove the covering of the halfspaces. We prove by assuming the contrary.
    
    Every point $a$ not covered by the halfspaces is in the intersection of the other halfspaces $\mathcal C=\bigcap\limits_{i} \{a: \langle a, x_i\rangle < 0\}$. Let $\mathcal C\cup B$ be contained in an open halfspace $\{a: \langle \vec{n}, a\rangle >0\}$ for some $\vec{n}$ (note that $\mathcal C\cup B\subseteq \bigcap\limits_{i} \{a: \langle a, x_i\rangle \le 0\}$). For any point $a$, we denote by $\mathcal P(a)$ the intersection between the ray $0a$ and the hyperplane $\{a: \langle \vec{n}, a\rangle=1\}$.
        
    For an uncovered normal $a$, the points in $\{a\}\cup B$ are in conical position since $a\notin\pos B$ by the definition of $B$ and each $a_j\in B$ is separated from $B\setminus\{a_j\}\cup\{a\}$ by the halfspace $\{a: \langle a, x_j\rangle\ge 0\}$.

    Among all the uncovered $a$, we consider a point $a^*$ with the corresponding $\mathcal P(a^*)$ having the minimal distance to $\conv \{\mathcal P(a_i): a_i\in B\}$. Such a point exists due to the closedness of $H$. We show that $\{a^*\}\cup B$ is a counterexample to the maximality of $B$. First, the points are in conical position as previously shown. Second, the positive hull is free of other points of $H$ as otherwise if there is any such point $a$, we obtain a smaller distance between $\mathcal P(a)$ and $\conv\{\mathcal P(a_i): a_i\in B\}$. Indeed, since $\{a^*\}\cup B$ are in conical position, $\{\mathcal P(a^*)\}\cup\{\mathcal P(a_i): a_i\in B\}$ are in convex position. Let $\mathcal P(a)$, which is in the convex hull of those points, be on the segment $\mathcal P(a^*)q$ for some $q\in\conv\{\mathcal P(a_i): a_i\in B\}$. Let $p$ be the point in $\conv\{\mathcal P(a_i): a_i\in B\}$ attaining the minimal distance to $\mathcal P(a^*)$. Consider the triangle $\mathcal P(a^*) p q$, the distance between $\mathcal P(a)$ (which is on $\mathcal P(a^*)q$) and the segment $pq$ is less than the length of $\mathcal P(a^*)p$, contradiction. (The case the triangle gets degenerated is trivial.)
\end{itemize}
  
  Being both the lower bound and the upper bound, the maximum of the Helly number and the cone number is actually the Carath\'eodory number for $H$-convexity.

\section{Proofs of related results}
The results discussed in Section \ref{sec:relations} are proved in this section.
\subsection{Proof of Theorem \ref{thm:lower-bound}}
\label{sec:lower-bound}
  Let us remind that given a supporting hyperplane $L$ through a regular boundary point $x$ of a convex body $K$, for every point $p$ in the same open halfspace as $K$, the homothety $K'$ of $K$ with center $x$ and any big enough ratio contains $p$. Note that $L$ is still a supporting hyperplane of $K'$.

  Let $h'$ be any number at most the Carath\'eodory number for $H$-convexity.\footnote{We consider this strange $h'$ instead of the Carath\'eodory number $h$ to deal with the case $h=\infty$.} We have then a set $X=\{x_1,\dots,x_{h'}\}$ such that $0\in\conv_H X$ but $0\notin\conv_H X'_i$ for any proper subset $X'_i=X\setminus\{x_i\}$.

  For each $i$, let $L$ be a hyperplane with a normal in $H$ strictly separating $0$ and $X'_i$. Translate $K$ to the position such that $L$ is a supporting hyperplane of $K$ through a regular boundary point $x$.
  The homothety $K'$ of $K$ with every sufficiently large ratio and center $x$ contains the whole finite $X'_i$ but not $0$. It means $0\notin\conv_{K'} X'$ for such $K'$.

  Since $X$ is finite, a homothety $K'$ of $K$ with a sufficiently large ratio will work for every $i$. That is for every $i$, we have $0\notin\conv_{K'} X'_i$ . However, $0\in\conv_H X\subseteq \conv_{K'} X$ (note that we also set $K'$ large enough so that a translate can contain $X$). It means the Carath\'eodory number for $K'$-strong convexity, which is the same as the one for $K$-strong convexity, is at least $h'$. The conclusion follows.

\subsection{Proof of Theorem \ref{thm:trivial-upper-bound}}
\begin{lemma}
\label{lem:uniq-guard}
If $p\in\conv_K X$ but $p\notin\conv_K X'$ for any proper subset $X'\subset X$, then for each $x\in X$ we can assign an $a\in H$ (where $H$ is the set of normals of $K$) such that $\langle a,x\rangle \ge \langle a,p\rangle$ and
\[
\langle a,x\rangle > \langle a,y\rangle
\]
for every $y\in X\setminus\{x\}$.
\end{lemma}
\begin{proof}
We assume $X\subseteq K$ without loss of generality.
For each $x\in X$, consider a translate $K-t$ such that $X\setminus\{x\}\subseteq K-t$ but both $p$ and $x$ are not in $K-t$. When we translate $K$ from $K$ to $K-t$, there exists a translate $K-t^*$ for $t^* = \alpha t$ ($0\le\alpha<1$) such that $x$ is on a supporting hyperplane $L$ with some normal $a$ of $K-t^*$ and the translate $K-t^*-\epsilon t$ (for any small enough $\epsilon > 0$) does not contain $x$. Every point $y$ in $X\setminus\{x\}$ is not on $L$, since otherwise $K-t^*-\epsilon t$ would not contain $y$.
The translate $K-t^*$ still contains $p$, since otherwise $K-t^*$ would contain the whole $X$ but not $p$, which implies $p\notin\conv_K X$, contradiction. The conclusion follows.
\end{proof}

Now we can prove Theorem \ref{thm:trivial-upper-bound} as a corollary of Lemma \ref{lem:uniq-guard}.
\begin{proof}[Proof of Theorem \ref{thm:trivial-upper-bound}]
Consider a point $p$ in $\conv_K X$ so that $p\notin\conv_K X'$ for any proper subset $X'\subset X$.

If $p$ is in $\conv_H X$, then
\[
p\in\conv_H X'\subseteq \conv_K X'
\]
for some subset $X'$ of at most $h$ points. That is $|X|\le h$.

It remains to show that $|X|$ is at most $|H|-1$ for the case $p$ is not in $\conv_H X$.
By Lemma \ref{lem:uniq-guard}, there exists an assignment of $a\in H$ to each point $x\in X$ such that $\langle a, x\rangle \ge \langle a, p\rangle$ and $\langle a, x\rangle > \langle a, y\rangle$ for every $y\in X\setminus\{x\}$. 
Since $p$ is not in $\conv_H X$, there must be a normal $a\in H$ such that $\langle a, x\rangle < \langle a, p\rangle$ for every $x\in X$. It means there remain at most $|H|-1$ normals of $H$ for the assignments to the points in $X$. In other words, $|X|\le |H|-1$ and the theorem follows.
\end{proof}

\subsection{Proof of Theorem \ref{thm:cara>helly}}
This proof is inspired by the proofs of the equivalence between Helly's theorem and Carath\'eodory's theorem in \cite[Chapter $2$]{eggleston1958convexity}.

Suppose the Helly number for the considered convex-structure is $h$. It means there are $h$ sets $S_1,\dots, S_h$ such that their intersection is empty but the intersection of every $h-1$ sets among them is nonempty.

Let $x_i\in \bigcap\limits_{j\ne i} S_j$ for each $i=1,\dots,h$, and let $\mathcal S=\conv\{x_1,\dots,x_h\}$. We show that $\mathcal S\setminus \bigcup\limits_i S_i$ is nonempty.

Let $p$ be the point in $\mathcal S$ minimizing the quantity $d=\max\limits_i d(p, S_i)$ where $d(p,S_i)$ is the distance from $p$ to $S_i$.
As $\bigcap\limits_i S_i=\emptyset$, the minimal $d$ is positive. We show that $p\notin S_i$ for every $i$. Assume otherwise, $p\in S_i$ for some $i$. Consider a slight translate $p'=p+\epsilon(x_i-p)$ of $p$ towards $x_i$ for a small enough $\epsilon>0$. Since $x_i\in\bigcap\limits_{j\ne i} S_i$, we have $d(p',S_j)<d(p,S_j)$ for every $j\ne i$ satisfying $d(p,S_j)>0$ (that is $p\notin S_j$). For those $j\ne i$ with $d(p,S_j)=0$, we still have $d(p',S_j)=0$. Although we may have a positive $d(p',S_i)$ which is greater than the zero $d(p,S_i)$, we still have $\max\limits_j d(p',S_j) < \max\limits_j d(p,S_j)$ since $\epsilon$ is small enough. As $p'$ is still in $\mathcal S$, we have a contradiction.

Denote by $\mathcal C$ the considered convex-structure and by $\conv_{\mathcal C}$ the corresponding convex hull. As $p\in\conv \{x_1,\dots,x_h\}$, we have $p\in\conv_{\mathcal C} \{x_1,\dots,x_h\}$. However, for every $i$, we have $\{x_1,\dots,x_h\} \setminus\{x_i\}\subseteq S_i$ and $p\notin S_i$, thus $p\notin \conv_{\mathcal C}(\{x_1,\dots,x_h\} \setminus\{x_i\}$). This example indicates that the Carath\'eodory number is at least the Helly number.

\section*{Acknowledgements}
The author would like to thank Roman Karasev for reading and commenting on many pieces of this paper, in particular, for his suggestion to the manipulation of halfspaces.

\bibliographystyle{unsrt}
\bibliography{carapoly}

\end{document}